\documentclass[12pt]{article}
\usepackage[margin=1.1in]{geometry}
\usepackage{subfigure}
\usepackage{amsmath}
\usepackage{amsthm} 
\allowdisplaybreaks
\usepackage{color}
\usepackage{footnote}
\usepackage{algorithm}
\usepackage{algpseudocode}
\usepackage{algorithmicx}
\usepackage{multirow} 
\usepackage{booktabs}
\usepackage{graphicx}
\usepackage{amssymb}
\usepackage{amsbsy}
\usepackage{array}
\usepackage{longtable}
\usepackage{epstopdf}
\usepackage{pbox}
\usepackage{breqn}
\usepackage{mathrsfs}
\usepackage{multicol}
\usepackage{supertabular}
\usepackage{enumerate}
\usepackage[colorinlistoftodos]{todonotes}
\usepackage{bbm}
\usepackage{bbold}
\usepackage{url}
\usepackage[utf8x]{inputenc}
\usepackage[bookmarks]{hyperref}
\usepackage{color}
\usepackage[normalem]{ulem}
\usepackage[justification=centering]{caption}

\newtheorem{dfn}{Definition}
\newtheorem{thm}{Theorem}

\newtheorem{prp}{Proposition}
\newtheorem{rem}{Remark}

\newtheorem{asm}{Assumption}

\newcommand{\Dbar}{\overline D}

\newcommand{\real}{\mathbb{R}} 

\newcommand{\I}{\mathcal{I}}

\newcommand{\Q}{\mathcal{Q}}

\newcommand{\sfp}{\mathsf{p}}

\newcommand{\diag}{\mathrm{diag}}
\newcommand{\Ell}{\mathcal{L}}

\newcommand{\bfe}{\mathbf{e}}

\begin{document}
To appear in \emph{IEEE Transactions on Automatic Control}

\begin{minipage}{\textwidth}

\title{Stability of Fluid Queueing Systems with Parallel Servers and Stochastic Capacities}

\author{Li~Jin
        and Saurabh~Amin
\thanks{L. Jin and S. Amin are with the Department
of Civil and Environmental Engineering, Massachusetts Institute of Technology, Cambridge, MA 02139 USA; e-mails: \{jnl,amins\}@mit.edu.}
\vspace{-0.6cm}
}

   \maketitle
\end{minipage}

\begin{abstract}
This note introduces a piecewise-deterministic queueing (PDQ) model to study the stability of traffic queues in parallel-link transportation systems facing stochastic capacity fluctuations. The saturation rate (capacity) of the PDQ model switches between a finite set of modes according to a Markov chain, and link inflows are controlled by a state-feedback policy. A PDQ system is stable only if a lower bound on the time-average link inflows does not exceed the corresponding time-average saturation rate. Furthermore, a PDQ system is stable if the following two conditions hold: the nominal mode's saturation rate is high enough that all queues vanish in this mode, and a bilinear matrix inequality (BMI) involving an underestimate of the discharge rates of the PDQ in individual modes is feasible. The stability conditions can be strengthened for two-mode PDQs. These results can be used for design of routing policies that guarantee stability of traffic queues under stochastic capacity fluctuations.   
\end{abstract}


\section{Introduction}

Capacity fluctuations in transportation systems can cause significant efficiency losses to the system operators~\cite{kwon06}. In practice, these fluctuations can be frequent and also hard to predict deterministically~\cite{peterson95,jin17}. Thus, traffic control strategies that assume fixed (or nominal) link capacities may fail to limit the inefficiencies resulting from capacity fluctuations, especially when their intensity and/or frequency is non-negligible. In this note, we introduce a simple fluid queueing model with parallel links (servers) that accounts for stochastically varying capacities of individual links, and investigate its stability under a class of feedback control policies. Our analysis is based on known results on stability of continuous-time Markov processes \cite{meyn93,cloez15}, and properties of piecewise-deterministic Markov processes (PDMPs) \cite{davis84,benaim15}.

Current literature on control of transportation systems with uncertain link or node capacities either assumes a static (but uncertain) capacity model, or considers time-varying capacities~\cite{peterson95, anick82, como13i}. In the former class of models, the actual capacity is assumed to lie in a known set~\cite{como13i}, or is realized according to a given probability distribution~\cite{glockner00}. Such models are useful for evaluating the system's performance against worst-case disturbances. The latter class of models is motivated by situations where the capacity is inherently dynamic. These models can enable more accurate assessment of system performance in comparison to static models. In contrast to the above two classes, the model considered in this note is applicable to situations where the capacity can be modeled as a Markovian process~\cite{peterson95,baykal09,jin14}.

Since we focus on the behavior of aggregate traffic flows, fluid queueing models are better suited to our objectives than the conventional queueing models (e.g. $M/M/1$) \cite{peterson95,newell13}. Single server fluid queueing systems with stochastically switching saturation rates have been studied previously; see~\cite{anick82,chen92,kulkarni97}. This line of work focuses on the analysis of the stationary distribution of queue length under a fixed inflow or an open-loop control policy. Some results are also available on feedback-controlled fluid queueing systems with stochastic capacities~\cite{peterson95,yu04}. However, to the best of our knowledge, stability of parallel-link fluid queueing systems with uncertain capacities has not been considered before.  

In Section~\ref{Sec_Model}, we introduce the parallel-link fluid queueing model with a stochastically switching saturation rate vector. This model is called the \emph{piecewise-deterministic queueing} (PDQ) model, since the saturation rate vector switches between a finite set of values, or {modes}, according to a Markov chain, while the evolution of queue lengths between mode switches is deterministic. Thus, our model belongs to the class of PDMPs~\cite{davis84}. An advantage of this model is that it can be easily calibrated using commonly available traffic data~\cite{jin17}. Furthermore, since the capacity and the queue lengths can be obtained using modern sensing technologies, this model can be used to design capacity-aware control policies. 

Our stability notion follows \cite{meyn93,benaim15} in that a PDQ system is {stable} if the joint distribution of its state (mode and queue lengths) converges to a unique {invariant probability measure}. Our analysis involves two assumptions: (i) the mode transition process is ergodic, and (ii) the feedback control policy is bounded and continuous in the queue lengths, and also satisfies a monotonicity condition to ensure that more traffic is routed through links with smaller queues. 

Under the above-mentioned assumptions, in Section~\ref{Sec_PDQ}, we derive a necessary condition (Theorem~\ref{Thm_PDQ1}) and a sufficient condition (Theorem~\ref{Thm_PDQ2}) for stability. The necessary condition is that, for every queue, a suitably defined lower bound on the time-average inflow does not exceed the corresponding link's time-average saturation rate. The sufficiency result requires two conditions: (i) all queues eventually vanish in a ``nominal'' mode; and (ii) a lower bound on the discharge rate of the system in individual modes verify a bilinear matrix inequality (BMI). Condition (i) essentially ensures the uniqueness of the invariant measure, and Condition (ii) sets a lower bound on the total flow discharged from the system. Theorem~\ref{Thm_PDQ2} also provides an exponential convergence rate towards the invariant measure. The sufficient conditions for stability can be verified in a more straightforward manner in the case when the PDQ system has two modes~(Proposition~\ref{prp_bpdq}). Furthermore, under a mode-responsive control policy the necessary and sufficient conditions for a two-mode PDQ coincide~(Proposition~\ref{prp_bimod}).

Finally, in Section~\ref{Sec_Net}, we illustrate some applications of our results for designing stabilizing traffic routing policies in parallel-link networks with stochastic capacity fluctuations.
\section{Piecewise-Deterministic Queueing System}
\label{Sec_Model}

Consider the PDQ system in Figure~\ref{Fig_PDQ} (left).
A constant \emph{demand} $A\ge0$ of traffic arrives at the system and is allocated to $n$ parallel servers.
The \emph{inflow} vector $F(t)=[F_1(t),\ldots,F_n(t)]^T\in\real_{\ge0}^n$ is such that $\sum_{k=1}^nF_k(t)=A$ for all $t\ge0$.
Traffic can be temporarily stored in queueing buffers and discharged downstream. We denote the vector of \emph{queue lengths} by $Q(t)=[Q_1(t),\ldots,Q_n(t)]^T$.
Let $U(t)=[U_1(t),\ldots,U_n(t)]^T$ denote the vector of stochastic \emph{saturation rates}, where $U_k(t)$ is the maximum rate at which the $k$-th server can release traffic at time $t$.

\begin{figure}[hbt]
\centering
\includegraphics[width=0.45\textwidth]{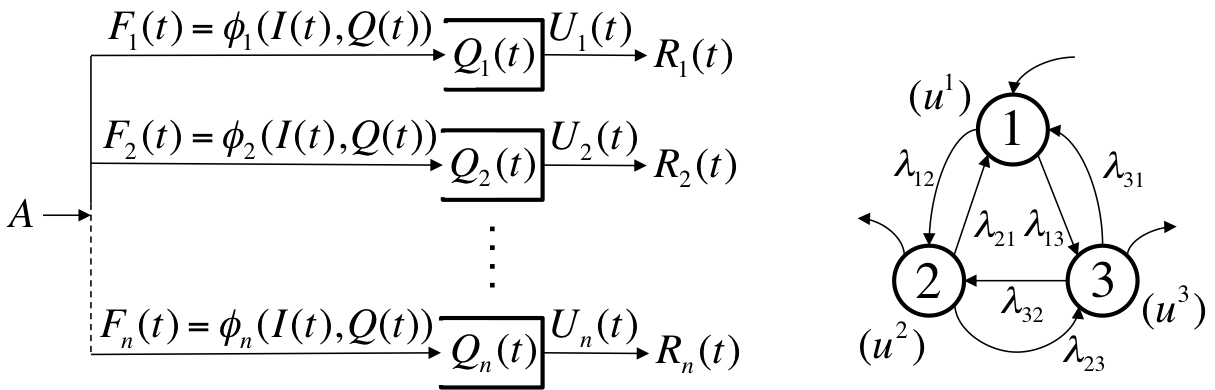}
\caption{Illustration of a PDQ system with $n$ parallel servers (left) and the mode transition process (right).}
\label{Fig_PDQ}
\end{figure}

For the $k$-th server, if $Q_k(t)=0$ and $F_k(t)\le U_k(t)$, the \emph{discharge rate} $R_k(t)$, i.e. the rate at which traffic departs from the system through the $k$-th server, is given by $R_k(t)=F_k(t)$; otherwise $R_k(t)=U_k(t)$.
We assume infinite buffer sizes; i.e. $Q(t)$ can take value in the set $\Q:=\real_{\ge0}^n$. This assumption enables us to account for all traffic arriving at the system and not just the traffic that is ultimately discharged by the system.

In our model, the saturation rates of the $n$ servers stochastically switch between a finite set of values.
To model this switching process, we introduce the set of \emph{modes} $\I$ of the PDQ system and let $m=|\I|$. We denote the mode of the PDQ system at time $t$ by $I(t)$. Each mode $i\in\I$ is associated with a {fixed} saturation rate, denoted by $u^i=[u_1^i,\ldots,u_n^i]^T$, which is distinct for each mode.
The evolution of $I(t)$ is governed by a finite-state Markov process with state space $\I$ and constant transition rates $\{\lambda_{ij};i,j\in\I\}$.
{\color{black}We assume that $\lambda_{ii}=0$ for all $i\in\I$.}
Note that this is without loss of generality, since self-transitions do not change the saturation rate; thus, including them will not affect the PDQ dynamics.
Let
\begin{align}
\nu_i:=\sum_{j\in\I}\lambda_{ij},
\label{Eq_nui}
\end{align}
which is the rate at which the system leaves mode $i$.
Given a fixed initial mode $I_0\in\I$ at $t=T_0:=0$, let $\{T_z;z=1,2,\ldots\}$ be the \emph{epochs} at which the mode transitions occur. Let $I_{z-1}$ be the mode during $[T_{z-1},T_z)$ and $S_z:=T_z-T_{z-1}$. Then, $S_z$ follows an exponential distribution with the cumulative distribution function (CDF):
\begin{align}
\mathsf F_{S_z}(s)=1-e^{-\nu_{I_{z-1}}s},\quad z=1,2\ldots
\label{Eq_Exp}
\end{align}
One can capture the transition rates in the $m\times m$ matrix:
\begin{align}
\Lambda:=\left[\begin{array}{llll}
-\nu_1 & \lambda_{12} &\ldots & \lambda_{1m}\\
\lambda_{21} & -\nu_2 &\ldots & \lambda_{2m}\\
\vdots & \vdots & \ddots & \vdots\\
\lambda_{m1} & \lambda_{m2} &\ldots & -\nu_m\\
\end{array}\right].
\label{Eq_Lmd}
\end{align}
{\color{black} We assume the following about the mode transition process:
\begin{asm}
The Markov process $\{I(t);t\ge0\}$ is \emph{ergodic}.
\label{Asm_Lambda}
\end{asm}
This assumption ensures that the process $\{I(t);t\ge0\}$ converges to a unique steady-state distribution, i.e. a row vector $\sfp=[\sfp_1,\ldots,\sfp_m]$ satisfying the following:
\begin{align}
\sfp\Lambda=0,\;|\sfp|=1,\;\sfp\ge0,
\label{Eq_pLmd}
\end{align}
where $|\cdot|$ is the $1$-norm.}

We consider that the demand $A$ is distributed across the $n$ servers according to a state-feedback \emph{routing policy}, which we denote as $\phi:\I\times\Q\to\real_{\ge0}^n$.
A routing policy is \emph{admissible} if $|\phi(i,q)|=A$ for all $(i,q)\in\I\times\Q$.
For a given routing policy $\phi$, the vector of discharge rates $R(t)$ is specified by the vector-valued function $r^\phi:\I\times\Q\to\real_{\ge0}^n$ with following components:
\begin{align}
&r^{\phi}_k(i,q):=\left\{\begin{array}{ll}
\phi_k(i,q), & q=0,\; \phi_k(i,q)\le u_k^i,\\
u_k^i, & \mbox{o.w.}
\end{array}\right.\nonumber\\
&\hspace{2in}k\in\{1,\ldots,n\};
\label{eq_r}
\end{align}
i.e., for each $t\ge0$, we have $R_k(t)=r^\phi_k(I(t),Q(t))\le U_k(t)$.

Let us define a vector field $D^\phi:\I\times\Q\to\real^n$ as follows:
\begin{align}
D^\phi(i,q):=\phi(i,q)-r^\phi(i,q).
\label{eq_D}
\end{align}
Then, the evolution of the \emph{hybrid state} $(I(t),Q(t))$ of the PDQ system is specified by the following dynamics:
\begin{subequations}
\begin{align}
&I(0)=i,\;Q(0)=q,\quad (i,q)\in\I\times\Q,\label{Eq_q0}\\
&\Pr\{I(t+\delta)=j'|I(t)=j\}=\lambda_{jj'}\delta+\mathrm o(\delta),\;j'\neq j,\\
&\frac{dQ(t)}{dt}=D^\phi\Big(I(t),Q(t)\Big),
\label{eq_Qdot}
\end{align}
\end{subequations}
where $\delta$ is an infinitesimal time increment.
Henceforth, we consider routing policies that satisfy the following assumption:
\begin{asm}
The routing policy $\phi(i,q)=[\phi_1(i,q),\ldots,\phi_n(i,q)]^T$ is bounded and continuous in $q$. 
Furthermore, for $k\in\{1,\ldots,n\}$, $\phi_k$ is non-increasing in $q_k$, and non-decreasing in $q_h$ for $h\neq k$.
\label{Asm_phi}
\end{asm}

The assumption of boundedness and continuity ensures that the Markov process $\{(I(t),Q(t));t\ge0\}$ is right continuous with left limits (RCLL, or \emph{c\`adl\`ag}) \cite{davis84}.
Furthermore, since $Q(t)$ is not reset after mode transitions, $Q(t)$ is necessarily continuous in $t$.
With the RCLL property, following \cite[Theorem 5.5]{davis84}, the \emph{infinitesimal generator} $\Ell^\phi$ of a PDQ with an admissible routing policy $\phi$ satisfying Assumption~\ref{Asm_phi} is given by
\begin{align}
&\Ell^\phi g(i,q)=\left(D^\phi(i,q)\right)^T\nabla_qg(i,q)\nonumber\\
&\quad+\sum_{j\in\I}\lambda_{ij}\Big( g(j,q)-g(i,q)\Big),
\quad(i,q)\in\I\times\Q,
\label{Eq_Lg}
\end{align}
where $g$ is any function on $\I\times\Q$ smooth in the continuous argument.

The assumption of monotonicity of controlled inflows with respect to queue lengths is practically relevant: more traffic is allocated to servers with smaller queues.
In addition, this assumption ensures the existence of the following limits:
\begin{align}
\varphi_{kh}^i:=\lim_{q_h\to\infty}\phi_k(i,q_h\bfe_h),\quad
h,k\in\{1,\ldots,n\},\;i\in\I,
\label{eq_varphi}
\end{align}
where $\bfe_h$ is the $n$-dimensional vector such that the $h$-th element is 1 and the others are 0.
Particularly, the monotonicity of $\phi$ also implies that $\phi_k(i,q)\ge\varphi_{kk}^i$ for all $k\in\{1,\ldots,n\}$ and all $(i,q)\in\I\times\Q$.

Many practically relevant routing policies satisfy Assumption~\ref{Asm_phi}. Examples include: 
\begin{enumerate}
	\item \emph{Mode-responsive} routing policy: 
\begin{align}
\phi_k^{\mathrm{mod}}(i)=\sum_{j\in\I}\mathbb1_{\{j=i\}}\psi_k^j,\quad
k\in\{1,\ldots,n\},
\label{eq_phimod}
\end{align}
where $\mathbb1_{\{\cdot\}}$ is the indicator function, and $\psi_k^i\ge0$ for $k\in\{1,\ldots,n\}$ and $i\in\I$.
This policy can be viewed as simple re-direction of traffic during disruptions.
	\item \emph{Piecewise-affine} routing policy:
\begin{align}
&\phi_k^{\mathrm{pwa}}(q)=\min\Big\{A,\Big(\theta_k-\alpha_{kk}q_k+\sum_{h\neq k}\alpha_{kh}q_h\Big)_+\Big\},\nonumber\\
&\hspace{1.75in} k\in\{1,\ldots,n\},
\label{eq_phipwa}
\end{align}
where $\alpha_{kh}\ge0$ for all $k,h\in\{1,\ldots,n\}$ and $(\cdot)_+$ indicates the positive part.
This policy is an example of a \emph{queue-responsive} traffic control policy.
Note that $\theta_k$ can be interpreted as the ``nominal'' inflow sent to each server when no queue exists throughout the system, and the linear terms $\alpha_{kh}q_h$ as adjustment to these inflows that accounts for the queue lengths. 
	\item \emph{Logit} routing policy:
\begin{align}
\phi_k^{\mathrm{log}}(q)=\frac{A\exp(\gamma_k-\beta_kq_k)}{\sum_{h=1}^n\exp(\gamma_h-\beta_hq_h)},\;k\in\{1,\ldots,n\},
\label{eq_philog}
\end{align}
where $\beta_k\ge0$ for $k\in\{1,\ldots,n\}$.
This is a classical model of travelers' route choice.
One can interpret $\beta_k$ as sensitivity parameter that reflects travelers' preference to the queue length in the $k$-th server, and $\gamma_k$ the parameter governing travelers' preference when every server has a zero queue.
\end{enumerate}
Note that the computation of the limiting inflows $\varphi_{kh}^i$ is rather straightforward for the above-mentioned routing policies (see Section~\ref{Sec_Net}).

Next, we introduce the notion of stability.
The \emph{transition kernel} \cite{meyn93} of a PDQ at time $t\ge0$ is a map $P_t$ from $\I\times\Q$ to the set of probability measures on $\I\times\Q$. Essentially, for an initial condition $(i,q)\in\I\times\Q$ and a measurable set $\mathcal E\subseteq\I\times\Q$, we have
\begin{align*}
P_t((i,q);\mathcal E)=\Pr\{(I(t),Q(t))\in \mathcal E|I(0)=i,Q(0)=q\}.
\end{align*}
One can also consider $P_t$ as an operator acting on probability measures $\mu$ on $\I\times\Q$ via
\begin{align}
\mu P_t(\mathcal E)=\int_{\I\times\Q}P_t((i,q);\mathcal E)d\mu.
\label{Eq_muPt}
\end{align}
An \emph{invariant probability measure} \cite{meyn93} of a PDQ system with routing policy $\phi$ is a probability measure $\mu_\phi$ such that
\begin{align*}
\mu_\phi P_t=\mu_\phi,\quad \forall t\ge0.
\end{align*}

\begin{dfn}[Stability \cite{cloez15,benaim15}]
The PDQ system with routing policy $\phi$ is \emph{stable} if there exists a probability measure $\mu_\phi$ on $\I\times\Q$ such that, for each initial condition $(i,q)\in\I\times\Q$, 
\begin{align}
\lim_{t\to\infty}\|P_t((i,q);\cdot)-\mu_\phi(\cdot)\|_{\mathrm{TV}}=0,\;
\forall(i,q)\in\I\times\Q,
\label{Eq_Stable}
\end{align}
where $\|\cdot\|_{\mathrm{TV}}$ is the total variation distance.
Furthermore, the PDQ system is \emph{exponentially stable} if it is stable and there exist constants $B>0$ and $c>0$, and a norm-like function\footnote{Following  \cite{meyn93}, $W$ is norm-like if $W(i,q)\to\infty$ as $\|q\|\to\infty$ for $i\in\I$.} $W:\I\times\Q\to[1,\infty)$ such that, for any $(i,q)\in\I\times\Q$,
\begin{align}
\|P_t((i,q);\cdot)-\mu_\phi(\cdot)\|_{\mathrm{TV}}\le BW(i,q)e^{-ct},\;\forall t\ge0.
\label{eq_exp}
\end{align}
\end{dfn}

Finally, the PDQ system is said to be \emph{unstable} if \eqref{Eq_Stable} does not hold.
\section{Stability of Feedback-Controlled PDQs}
\label{Sec_PDQ}

In this section, we study the stability of controlled PDQ systems.
Our main results are Theorem~\ref{Thm_PDQ1} (a necessary condition for stability) and Theorem~\ref{Thm_PDQ2} (a sufficient condition for stability).

\begin{thm}
\label{Thm_PDQ1}
Suppose that a PDQ system with $n$ parallel servers is subject to a total demand $A\in\real_{\ge0}$ and is controlled by an admissible policy $\phi$. 
If the PDQ system is stable, then
\begin{align}
\sum_{i\in\I}\sfp_i\varphi_{kk}^i
\le\sum_{i\in\I}\sfp_iu_k^i,\;
k\in\{1,\ldots,n\},
\label{eq_stable1}
\end{align}
where $\sfp_i$ are given by \eqref{Eq_pLmd} and $\varphi_{kk}^i$ are given by \eqref{eq_varphi}.
\end{thm}

\begin{proof}

Suppose that the PDQ system is stable.

For each server $k\in\{1,\ldots,n\}$ and for each initial condition $(i,q)\in\I\times\Q$, we obtain from \eqref{eq_D} and \eqref{eq_Qdot} that, for all $t\ge0$,
\begin{align*}
Q_k(t)=\int_{0}^t\left(\phi_k(I({s}),Q({s}))-r_k^\phi(I({s}),Q({s}))\right)d{s}+q_k.
\end{align*}
Since $\lim_{t\to\infty}q_k/t=0$, we have
\begin{align}
0&=\lim_{t\to\infty}\frac{1}{t}\Bigg(\int_{0}^t\left(\phi_k(I({s}),Q({s}))-r_k^\phi(I({s}),Q({s}))\right)d{s}\nonumber\\
&\hspace{.75in}+q_k-Q_k(t)\Bigg)\nonumber\\
&=\lim_{t\to\infty}\frac{1}{t}\Bigg(\int_{0}^t\Big(\phi_k(I({s}),Q({s}))-r_k^\phi(I({s}),Q({s}))\Big)d{s}\nonumber\\
&\hspace{.75in}-Q_k(t)\Bigg).
\label{Eq_limphi}
\end{align}

Since the $k$-th queue is stable, for each initial condition $(i,q)\in\I\times\Q$, $\Pr\{\lim_{t\to\infty}Q(t)=\infty\}=0$ (i.e. \emph{non-evanescence}, see \cite[pp. 524]{meyn93} for details), and we have $\lim_{t\to\infty}Q_k(t)/t=0$ a.s.
Hence, we can rewrite \eqref{Eq_limphi} as
\begin{align*}
\lim_{t\to\infty}\frac{1}{t}\int_{0}^t\Big(\phi_k(I({s}),Q({s}))-r_k^\phi(I({s}),Q({s}))\Big)d{s}=0,\;a.s.
\end{align*}
Now we can make two observations.
First, by Assumption~\ref{Asm_phi} (monotonicity), we have 
\begin{align*}
\phi_k(I({s}),Q({s}))
&\ge\phi_k(I({s}),Q_k({s})\bfe_k)\\
&\ge\lim_{q_k\to\infty}\phi_k(I({s}),q_k\bfe_k)=\varphi_{kk}^{I(s)},
\;\forall{s}\ge0.
\end{align*}
Secondly, recall that \eqref{eq_r} implies $r_k^\phi(I({s}),Q({s}))\le U_k({s})$ for ${s}\ge0$.
Thus, we have
\begin{align}
0&=\lim_{t\to\infty}\frac{1}{t}\int_{0}^t\Big(\phi_k(I({s}),Q({s}))-r_k^\phi(I({s}),Q({s}))\Big)d{s}\nonumber\\
&\ge\lim_{t\to\infty}\frac{1}{t}\int_{0}^t\Big(\varphi_{kk}^{I({s})}-U_k({s})\Big)d{s}.
\label{Eq_limphi2}
\end{align}

In addition, for every $i\in\I$, let $M_i(t)$ be the amount of time that the PDQ system is in mode $i$ up to time $t$, i.e.:
\begin{align*}
M_i(t)=\int_{0}^t\mathbb1_{\{I({s})=i\}}d{s}.
\end{align*}
Then, under Assumption~\ref{Asm_Lambda}, we have
\begin{align*}
\lim_{t\to\infty}\frac{M_i(t)}{t}=\sfp_i,\;a.s.\;\forall i\in\I.
\end{align*}
Hence,
\begin{align}
&\lim_{t\to\infty}\frac{1}{t}\int_{0}^t\varphi_{kk}^{I({s})}d{s}=\lim_{t\to\infty}\frac{1}{t}\int_{0}^t\Big(\sum_{i\in\I}\mathbb1_{\{I({s})=i\}}\varphi_{kk}^i\Big)d{s}\nonumber\\
&=\lim_{t\to\infty}\sum_{i\in\I}\frac{M_i(t)}{t}\varphi_{kk}^i
=\sum_{i\in\I}\sfp_i\varphi_{kk}^i,\;a.s.
\label{Eq_lim2}
\end{align}
Similarly, we can obtain
\begin{align}
\lim_{t\to\infty}\frac{1}{t}\int_{0}^tU_k({s})d{s}
=\sum_{i\in\I}\sfp_iu_k^i,\;a.s.
\label{Eq_lim3}
\end{align}

Combining \eqref{Eq_limphi2}--\eqref{Eq_lim3}, we obtain \eqref{eq_stable1}.
\end{proof}

Theorem~\ref{Thm_PDQ1} provides a way of identifying unstable control policies.
As argued in the proof, $\varphi_{kk}^i$ is in fact the lower bound for $\phi_k(i,q)$ for all $q\in\Q$.
Hence, Theorem~\ref{Thm_PDQ1} essentially states that if the PDQ system is stable, then the (time-average) lower bound of the inflow does not exceed the average saturation rate.

To introduce our next result, we define $R_\min=[R_\min^1,\ldots,R_\min^m]^T$ as follows:
\begin{align}
	R_\min^i=\min_{k}\Bigg(u_k^i+\sum_{\substack{h:h\neq k}}\min\{u_h^i,\varphi_{hk}^i\}\Bigg),\quad i\in\I.
	\label{eq_Rmin}
\end{align}
One can interpret $R_\min^i$ as a lower bound on the total discharge rate of the $n$ servers in mode $i$ when at least one of the $n$ servers has a non-zero queue.
Our next result uses $R_\min^i$ to provide a sufficient condition for the stability of feedback-controlled PDQ systems.

\begin{thm}
\label{Thm_PDQ2}
Suppose that a PDQ system of $n$ parallel servers is subject to a total demand $A\in\real_{\ge0}$ and is controlled by an admissible policy $\phi$.
Let the elements of the vector $R_\min$ be as defined in \eqref{eq_Rmin}.
Then, the PDQ system is stable if
\begin{align}
\exists i^*\in\I,\;\forall k\in\{1,\ldots,n\},\quad
	\phi_k(i^*,0)<u_k^{i^*},
	\label{eq_phi0<u}
\end{align}
and if
\begin{align}
&\exists a=[a^1,\ldots,a^m]^T\in\real_{>0}^m,\;\exists b>0,\nonumber\\
	&\hspace{1in}\Big(\diag(A\bfe-R_\min)b+\Lambda\Big)a\le-\bfe,
	\label{eq_BMI}
\end{align}
where $\bfe$ is the $m$-dimensional vector of 1's.
Furthermore, under the above conditions, there exists a positive constant $c=\min_{i\in\I}1/(2a^i)$ such that, for some $B>0$,
\begin{align}
&\|P_t((i,q);\cdot)-\mu_\phi(\cdot)\|_{\mathrm{TV}}\le B\Big(a^{i} e^{b|q|}+1\Big)e^{-ct},\nonumber\\
&\hspace{1.5in}\forall(i,q)\in\I\times\Q,\;\forall t\ge0,
\label{Eq_ect}
\end{align}
where $\mu_\phi$ is the unique invariant probability measure.
\end{thm}

The proof of Theorem~\ref{Thm_PDQ2} is based on a more general result \cite[Theorem 6.1]{meyn93}, which we recall here in the setting of PDQ systems.
To conclude stability of the PDQ system, \cite[Theorem 6.1]{meyn93} requires that the following two conditions hold:
\begin{enumerate}[(A)]
\item For any two initial conditions $(i,q),(j,\ell)\in\I\times\Q$, there exist $\delta>0$ and $T>0$ such that
\begin{align}
\|P_T((i,q);\cdot)-P_T((j,\ell);\cdot)\|_{\mathrm{TV}}\le1-\delta.
\label{eq_TV}
\end{align}
\item There exist a norm-like function $V:\I\times\Q\to\real_{\ge0}$ (called the \emph{Lyapunov function}) and constants $c>0$ and $d<\infty$ such that
\begin{align}
\Ell V(i,q)\le-cV(i,q)+d,\quad\forall(i,q)\in\I\times\Q.
\label{Eq_Drift}
\end{align}
\end{enumerate}
Condition (A) is required for the uniqueness of the invariant probability measure \cite{dai95ii}.
Condition (B) is usually referred to as the \emph{drift condition}, which essentially ensures the existence of invariant probability measures \cite[Theorem 4.5]{meyn93}.

We are now ready to prove the theorem:

\begin{proof}[Proof of Theorem~\ref{Thm_PDQ2}]

Suppose that \eqref{eq_phi0<u} and \eqref{eq_BMI} hold.
We verify condition (A) (resp. (B)) using \eqref{eq_phi0<u} (resp. \eqref{eq_BMI}).

\noindent\underline{Condition (A)}:

Consider any initial condition $(i_0,q_0)\in\I\times\Q$.

First, Assumption~\ref{Asm_Lambda} ensures that the Markov process $\{I(t),Q(t);t\ge0\}$ recurrently visits the mode $i^*$. That is, for any $X_1>0$, there exists $\sigma>0$ such that
\begin{align}
\Pr\{I(X_1)=i^*|I(0)=i_0,Q(0)=q_0\}=\sigma.
\label{eq_delta}
\end{align}
Furthermore, we can obtain from \eqref{eq_Qdot} that
\begin{align}
|Q(X_1)|
&=\Big|q_0+\int_0^{X_1}\Big(A-\sum_{k=1}^nr_k^\phi(I({s}),Q({s}))\Big)d{s}\Big|\nonumber\\
&\le|q_0|+\int_0^{X_1}\Big|A-\sum_{k=1}^nr_k^\phi(I({s}),Q({s}))\Big|d{s}\nonumber\\
&\le|q_0|+AX_1.
\label{eq_Q<AX1}
\end{align}

Secondly, in mode $i^*$, the vector of queue length $Q(t)$ necessarily converges to $q^*=0$.
To see this, consider mode $i^*$ and any $q\in\Q$. For each $k\in\{1,\ldots,n\}$ such that $q_k=0$, by Assumption~\ref{Asm_phi}, we have
$\phi_k(i^*,q)\ge\phi_k(i^*,0)$, and thus
\begin{align}
r_k^\phi(i^*,q)&=\min\{u_k^{i^*},\phi_k(i^*,q)\}\nonumber\\
&\ge\min\{u_k^{i^*},\phi_k(i^*,0)\}
= r_k^\phi(i^*,0).
\label{eq_rq>r0}
\end{align}
Therefore, for each $q\in\Q\backslash\{0\}$, we have
\begin{align}
&\sum_{k=1}^nD^\phi_k(i^*,q)
\stackrel{\footnotesize\eqref{eq_D}}=A-\sum_{k=1}^nr_k^\phi(i^*,q)\nonumber\\
&\quad=A-\sum_{k:q_k>0}u_k^{i^*}-\sum_{k:q_k=0}r_k^\phi(i^*,q)\nonumber\\
&\quad\stackrel{\footnotesize\eqref{eq_rq>r0}}{\le} A-\sum_{k:q_k>0}u_k^{i^*}-\sum_{k:q_k=0}r_k^\phi(i^*,0)\nonumber\\
&\quad\le A-\min_{k:q_k>0}\Big(u_k^{i^*}+\sum_{\substack{h:q_h>0\\h\neq k}}r_h^\phi(i^*,0)\Big)-\sum_{k:q_k=0}r_k^\phi(i^*,0)\nonumber\\
&\quad{\le} A-\min_{k\in\{1,\ldots,n\}}\Big(u_k^{i^*}+\sum_{h\neq k}r_h^\phi(i^*,0)\Big)\nonumber\\
&\quad\stackrel{\footnotesize\eqref{eq_phi0<u}}<A-\sum_{k=1}^nr_k^\phi(i^*,0)
\stackrel{\footnotesize\eqref{eq_r}\eqref{eq_phi0<u}}{=}0.
\label{eq_D<0}
\end{align}
One can see from \eqref{eq_Q<AX1} and \eqref{eq_D<0} that there exists
\begin{align*}
X_2=\frac{|q_0|+AX_1}{A-\min_{k}\Big(u_k^{i^*}+\sum_{h\neq k}r_h^\phi(i^*,q)\Big)}
\end{align*}
such that $Q(X_1+X_2)=0$ if $I(t)=i^*$ for all $t\in[X_1,X_2+X_2)$. Note that
\begin{align*}
\Pr\{I(t)=i^*;t\in[X_1,X_1+X_2)|I(X_1)=i^*\}
=e^{-\nu_{i^*}X_2}.
\end{align*}

Thus, we have
\begin{align*}
\Pr\{Q( X_1+ X_2)=0|I(0)=i_0,Q(0)=q_0\}
\ge\sigma e^{-\nu_{i^*} X_2}>0,
\end{align*}
where $\sigma$ satisfies \eqref{eq_delta}.
Hence, we have
\begin{align*}
P_{X_1+X_2}((i_0,q_0),\{(i^*,0)\})\ge\sigma e^{-\nu_{i^*} X_2}.
\end{align*}
Then, for any $T\ge X_1+X_2$, we have
\begin{align*}
P_{T}((i_0,q_0),\{(i^*,0)\})\ge\sigma e^{-\nu_{i^*} (T-X_1)}.
\end{align*}

Thus, for arbitrary initial conditions $(i,q)$ and $(j,\ell)$, there exist $\sigma'>0$, $X_1'>0$, $X_2'>0$, and $T'>0$ such that
\begin{align*}
&P_{T'}((i,q),\{(i^*,0)\})\ge\sigma' e^{-\nu_{i^*} (T'-X_2')},\\
&P_{T'}((j,\ell),\{(i^*,0)\})\ge\sigma' e^{-\nu_{i^*} (T'-X_2')},
\end{align*}
which verifies \eqref{eq_TV} with $T=T'$ and $\delta=\sigma' e^{-\nu_{i^*} (T'-X_2')}$.

\noindent\underline{Condition (B)}:

Consider the Lyapunov function
\begin{align}
V(i,q)=a^i e^{b|q|},\quad (i,q)\in\I\times\Q,
\label{Eq_V2}
\end{align}
where $a^1,\ldots,a^m$, and $b$ are positive constants.

For each server $k$, by the definition of $\varphi_{kh}^i$ \eqref{eq_varphi}, there necessarily exists $L_k<\infty$ such that, for all $h\neq k$,
\begin{align}
\min\Big\{u_h^i,\phi_h(i,L_k\bfe_k)\Big\}\ge\min\Big\{u_h^i,\varphi_{hk}^i\Big\}-\frac1{2nb\max_{j\in\I}a^j}.
\label{eq_Lk}
\end{align}
Let $L=[L_1,\ldots,L_n]^T$.
We claim that the constants
\begin{subequations}
\begin{align}
&c:=\frac1{2\max_{j\in\I}a^j},\label{Eq_c}\\
&d:=\max_{i\in\I}|\Ell V(i,L)+cV(i,L)|,\label{Eq_d}
\end{align}
\label{Eq_cd}
\end{subequations}
\hspace{-3pt}verify the drift condition \eqref{Eq_Drift}.
Let us prove this claim.

Plugging the Lyapunov function defined in \eqref{Eq_V2} into the expression of the infinitesimal generator \eqref{Eq_Lg}, we obtain
\begin{align}
\Ell V(i,q)&=\Bigg(\sum_{k=1}^n\left(\phi_k(i,q)-r^\phi_k(i,q)\right)a^i b\nonumber\\
&\quad\quad+\sum_{j\in\I}\lambda_{ij}(a^j-a^i)\Bigg){ e^{b|q|}}.
\label{Eq_LV2}
\end{align}
Then, to check \eqref{Eq_Drift}, we need to consider two cases:

\emph{Case I: $q\in\{\zeta\in\Q:0\le \zeta\le L\}$.} Since each such $q$'s are bounded, $V(i,q)$ is also bounded. Hence, we can verify in a rather straightforward manner that, with $c$ and $d$ given by \eqref{Eq_cd}, $\Ell V\le -cV+d$ for all $i\in\I$ and $0\le q\le L$.

\emph{Case II: $q\in\Q\backslash\{\zeta\in\Q:0\le \zeta\le L\}$.} For each such $q$, there necessarily exists a server $k_1$ such that $q_{k_1}>L_{k_1}$.
For the ${k_1}$-th server, since $q_{k_1}>L_{k_1}\ge0$, we have
\begin{align}
r_{k_1}^\phi(i,q)=u_{k_1}^i,\;\forall i\in\I.
\label{eq_rkphi}
\end{align}

For the other servers, i.e. for each $h\neq {k_1}$, we have
\begin{subequations}
\begin{align}
&r^\phi_h(i,q)
\stackrel{\footnotesize\eqref{eq_r}}=\min\Big\{u_h^i,\phi_h(i,q)\Big\}\nonumber\\
&\ge\min\{u_h^i,\phi_h(i,q_{k_1}\bfe_{k_1})\}
\ge\min\{u_h^i,\phi_h(i,L_{k_1}\bfe_{k_1})\}\label{eq_>=min2}\\
&\stackrel{\footnotesize\eqref{eq_Lk}}\ge\min\{u_h^i,\varphi_h^{k_1}(i)\}-\frac1{2nb\max_{j\in\I}a^j},\;\forall i\in\I,
\label{eq_rhphi}%
\end{align}
\end{subequations}
where \eqref{eq_>=min2} results from Assumption~\ref{Asm_phi} (monotonicity).
Combining \eqref{eq_rkphi} and \eqref{eq_rhphi}, we can write
\begin{align}
\sum_{h=1}^nr^\phi_h(i,q)
&\ge u_{k_1}^i+\sum_{\substack{h:h\neq {k_1}}}\min\Big\{u_h^i,\varphi_h^{k_1}(i)\Big\}-\frac1{2b\max_{j\in\I}a^j}\nonumber\\
&\stackrel{\footnotesize\eqref{eq_Rmin}}\ge R_\min^i-\frac1{2b\max_{j\in\I}a^j}.
\label{eq_sumrhphi}
\end{align}
Then,
\begin{align*}
&\sum_{k=1}^n\left(\phi_k(i,q)-r^\phi_k(i,q)\right)a^i b+\sum_{j\in\I}\lambda_{ij}(a^j-a^i)\\
&\quad\stackrel{\footnotesize\eqref{eq_sumrhphi}}{\le}\left(A-R_\min^i+\frac1{2b\max_{j\in\I}a^j}\right)a^i b+\sum_{j\in\I}\lambda_{ij}(a^j-a^i)\\
&\quad\stackrel{\footnotesize\eqref{eq_BMI}}{\le}-1+\frac12=-\frac12.
\end{align*}
Finally,
\begin{align*}
\Ell V(i,q)\stackrel{\footnotesize\eqref{Eq_LV2}}{\le}-\frac12 e^{b|q|}
\stackrel{\footnotesize{\eqref{Eq_c}}}{\le}-ca^i e^{b|q|}\stackrel{\footnotesize\eqref{Eq_V2}}{=}-cV.
\end{align*}
Hence, \eqref{Eq_Drift} holds for all $i\in\I$, all $q\in\Q\backslash\{q:0\le \zeta\le L\}$, and all $d\ge0$.

Thus, we have verified that the drift condition \eqref{Eq_Drift} holds for all $(i,q)\in\I\times\Q$.

Finally, note that we have verified conditions (A) and (B) for the controlled PDQ system. Thus, we obtain from \cite[Theorem 6.1]{meyn93} that the PDQ system is exponentially stable.

\end{proof}

The condition \eqref{eq_phi0<u} states that there exists a mode $i^*$ in which every queue decreases to zero.
Practically, one can interpret $i^*$ as a ``nominal'' or ``normal'' mode in which the saturation rates are sufficiently high and satisfy \eqref{eq_phi0<u}.
This condition leads to Condition (A).

The condition \eqref{eq_BMI} essentially imposes a lower bound on the total discharged flow from the $n$ servers, which is characterized by $R_\min^i$.
This condition leads to Condition (B).
To verify this condition, one needs to determine whether BMI \eqref{eq_BMI} admits positive solutions for $a_1,\ldots,a_m$ and $b$. This can be done using the known computational methods to solve BMIs (see e.g. \cite{vanantwerp00,lofberg04}). 

\begin{rem}
Using the exponential Lyapunov function (31), one can also apply \cite[Theorem 4.3]{meyn93} to obtain that, under (23), for each initial condition $(i,q)\in\I\times\Q$, we have
\begin{align*}
\limsup_{t\to\infty}\frac1t\int_0^t\mathsf E[e^{|Q(s)|}]ds<\infty.
\end{align*}
That is, moments of the queue lengths are bounded.
\end{rem}

Furthermore, if the system has only two modes, solutions for $b$ and $a$ can be constructed in a more straightforward manner, which motivates the next result.
\begin{prp}
A PDQ system of $n$ parallel servers with two modes $\{1,2\}$ and with an admissible control policy $\phi$ is stable if
\begin{align}
\exists i^*\in\{1,2\},\;
	\phi_k(i^*,0)<u_k^{i^*},\;
	k\in\{1,2\},
	\label{eq_phi0<u2}
\end{align}
and if
\begin{align}
A<\sfp_1R_\min^1+\sfp_2R_\min^2,
\label{eq_A<R}
\end{align}
where $R_\min^i$ is defined in \eqref{eq_Rmin}.
\label{prp_bpdq}
\end{prp}

\begin{proof}

First, let us define the following quantities
\begin{subequations}
\begin{align}
&D_\min=\min\left\{A-R_\min^1,A-R_\min^2\right\},\label{Eq_Fmin}\\
&D_\max=\max\left\{A-R_\min^1,A-R_\min^2\right\},\label{Eq_Fmax}\\
&\Dbar=A-(\sfp_1R_\min^1+\sfp_2R_\min^2),\\
&i_{\min}=\left\{\begin{array}{ll}
1, & \mbox{if } D_\min=A-R_\min^1,\\
2, & \mbox{o.w.}
\end{array}\right.\label{eq_imin}\\
&i_{\max}=\left\{\begin{array}{ll}
2, & \mbox{if } D_\min=A-R_\min^1,\\
1, & \mbox{o.w.}
\end{array}\right.\label{eq_imax}\\
&\lambda_{\min}=\left\{\begin{array}{ll}
\lambda_{12}, & \mbox{if } D_\min=A-R_\min^1,\\
\lambda_{21}, & \mbox{o.w.}
\end{array}\right.\label{Eq_lambdatilde}\\
&\lambda_{\max}=\left\{\begin{array}{ll}
\lambda_{21}, & \mbox{if } D_\min=A-R_\min^1,\\
\lambda_{12}, & \mbox{o.w.}
\end{array}\right.\label{Eq_mutilde}
\end{align}
\label{Eq_tildes}
\end{subequations}

Under \eqref{eq_A<R}, we explicitly construct constants $a^{i_\min}$, $a^{i_\max}$, and $b$ satisfying the BMI \eqref{eq_BMI}. 
Condition \eqref{eq_A<R} implies
\begin{align}
A-\sfp_1R_\min^1-\sfp_2R_\min^2
=\sfp_{i_\min}D_\min+\sfp_{i_\max}D_\max<0
\label{Eq_phibarinf-ubar}
\end{align}
Since $D_\min\le D_\max$, \eqref{Eq_phibarinf-ubar} implies that $D_\min<0$.
Thus, we only need to consider two cases:

\emph{In the case that} $D_\min<0,\;D_\max\le0$, we can select an arbitrary $a^{i_\min}>\max_i\{1/\lambda_i\}$ and let
\begin{align}
a^{i_\max}=2a^{i_\min},\;
b=\frac{\lambda_\min a^{i_\min}+1}{-D_\min a^{i_\min}}.
\label{Eq_ab1}
\end{align}
It is not hard to see that $a^{i_\min}$, $a^{i_\max}$, and $b$ are positive and satisfy the BMI \eqref{eq_BMI}.

\emph{In the case that} $D_\min<0$, $D_\max>0$, we let
\begin{subequations}
\begin{align}
&b=\frac{(\lambda_{12}+\lambda_{21}){\Dbar}}{2D_\min D_\max},\label{Eq_b}\\
&a^{i_\min}=\frac{-D_\max b+\lambda_{12}+\lambda_{21}}{\det[\diag(A\bfe-R_\min)b+\Lambda]},\\
&a^{i_\max}=\frac{-D_\min b+\lambda_{12}+\lambda_{21}}{\det[\diag(A\bfe-R_\min)b+\Lambda]}.\label{Eq_a2}
\end{align}
\label{Eq_ab2}
\end{subequations}

Now, we show that these constants are positive.
First, note that \eqref{eq_A<R} implies $\Dbar<0$.
Then, since $D_\min<0$ and $D_\max>0$, and since $\Dbar<0$, $b$ is positive.
Secondly, to see that $a^{i_\min}>0$, note that
\begin{align*}
&\diag(A\bfe-R_\min)b+\Lambda\\
&={\small\left[\begin{array}{cc}
b(A-R_\min^1)-\lambda_{12}&\lambda_{12}\\
\lambda_{21}&b(A-R_\min^2)-\lambda_{21}
\end{array}\right]},
\end{align*}
and
\begin{align*}
&\mathrm{det}[\diag(A\bfe-R_\min)b+\Lambda]\nonumber\\
&=b^2\left(A-R_\min^1\right)\left(A-R_\min^2\right)\\
&\hspace{.5in}-\lambda_{12} b\left(A-R_\min^2\right)-\lambda_{21} b\left(A-R_\min^1\right)\nonumber\\
&=b^2\left(A-R_\min^1\right)\left(A-R_\min^2\right)
-b(\lambda_{12}+\lambda_{21}){\Dbar}\\
&=b^2D_\min D_\max-b(\lambda_{12}+\lambda_{21}){\Dbar}.
\end{align*}
Again, since $D_\min<0$ and $D_\max>0$, one can check that the $b$ given in \eqref{Eq_b} ensures that $\det[\diag(A\bfe-R_\min)b+\Lambda]>0$.
In addition, note that
\begin{align*}
b
&=\frac{(\lambda_{12}+\lambda_{21})\Dbar}{2D_\min D_\max}\\
&=\frac{\lambda_{12}+\lambda_{21}}{D_\max}\left(\frac{-\sfp_{i_\min}D_\min-\sfp_{i_\max}D_\max}{-2D_\min}\right)\\
&<\frac{\lambda_{12}+\lambda_{21}}{D_\max}\left(\frac{-\sfp_{i_\min}D_\min-\sfp_{i_\max}D_\min}{-2D_\min}\right)\\
&=\frac{\lambda_{12}+\lambda_{21}}{2D_\max}
<\frac{\lambda_{12}+\lambda_{21}}{D_\max},
\end{align*}
which, along with $D_\max>0$, implies $a^{i_\min}>0$.
Finally, since $D_\min<0$, $a^{i_\max}$ is also positive. 

From \eqref{eq_imin} and \eqref{eq_imax}, we know that
\begin{align*}
&a^1=\left\{\begin{array}{ll}
a^{i_\min}, & \mbox{if } D_\min=A-R_\min^1,\\
a^{i_\max}, & \mbox{o.w.}
\end{array}\right.\\
&a^2=\left\{\begin{array}{ll}
a^{i_\max}, & \mbox{if } D_\min=A-R_\min^1,\\
a^{i_\min}, & \mbox{o.w.}
\end{array}\right.
\end{align*}
Let $a=[a^1,a^2]^T$.
Then, one can check that $a$ and $b$ satisfy
\begin{align*}
[\diag(A\bfe-R_\min)b+\Lambda]a=-\bfe,
\end{align*}
and thus satisfy the BMI \eqref{eq_BMI}.

In addition, \eqref{eq_phi0<u2} is analogous to \eqref{eq_phi0<u}.
Thus, we can conclude from Theorem~\ref{Thm_PDQ2} that the two-mode PDQ system is stable.

\end{proof}

In comparison to Theorem~\ref{Thm_PDQ2}, Proposition~\ref{prp_bpdq} provides a simpler criterion \eqref{eq_A<R} for stability of PDQ systems with two modes, since it does not involve solving a BMI.

Furthermore, if a PDQ system with two modes is controlled by a mode-responsive routing policy \eqref{eq_phimod}, then we can obtain a \emph{necessary and sufficient condition} for stability:

\begin{prp}
\label{prp_bimod}
A system of $n$ parallel servers two modes $\{1,2\}$ and with a mode-responsive routing policy $\phi$ given by \eqref{eq_phimod} is stable if and only if
\begin{align}
\sfp_1\psi_k^1+\sfp_2\psi_k^2<\sfp_1u_k^1+\sfp_2u_k^2,\;\forall k\in\{1,\ldots,n\}.
\label{eq_phi<u}
\end{align}
\end{prp}

\begin{proof}
Since the system is controlled by a mode-responsive policy, the queues in various servers do not interact.
Therefore, we can consider the $n$ servers independently.
For the $k$-th server, consider the Lyapunov function
\begin{align*}
V_k(i,q_k)=a_k^i\exp(b_kq_k),\;(i,q)\in\{1,2\}\times\real_{\ge0}
\end{align*}
with parameters $[a_k^1,a_k^2]^T\in\real_{>0}^2$ and $b_k>0$.
With this Lyapunov function, one can adapt the proof of Proposition~\ref{prp_bpdq} and conclude that the $k$-th server is stable if \eqref{eq_phi<u} holds.

To obtain the necessity of \eqref{eq_phi<u}, first note that the $k$-th server is unstable if $
\sfp_1\psi_k^1+\sfp_2\psi_k^2>\sfp_1u_k^1+\sfp_2u_k^2.$
Secondly, to argue that the $k$-th server is unstable if
\begin{align}
\sfp_1\psi_k^1+\sfp_2\psi_k^2=\sfp_1u_k^1+\sfp_2u_k^2,
\label{eq_equality}
\end{align}
one can first assume by contradiction the existence of an invariant probability measure $\mu_\phi$, and then consider $\mu_\phi(\I\times\{0\})$ to arrive at a contradiction to \eqref{eq_equality}.
\end{proof}

In addition, for the setting of Proposition \ref{prp_bimod}, expression for the invariant probability measure $\mu_\phi$ has been reported in the literature \cite{kulkarni97}, which makes possible analytical optimization of the routing policy.
\section{Illustrative examples}
\label{Sec_Net}

In this section, we demonstrate how our results can provide insights for traffic flow routing under stochastic capacity fluctuation. 
Consider a network of two parallel servers.
The total inflow is $A=1$.
Our results in Section~\ref{Sec_PDQ} can be applied to obtain stability conditions of this network.
We particularly focus on the practically motivated routing policies given in \eqref{eq_phimod}--\eqref{eq_philog}.

\subsection{A two-mode network}

Suppose that the network has two modes $\{1,2\}$ with symmetric transition rates $\lambda_{12}=\lambda_{21}=1$. Thus, the steady-state probabilities are $\sfp_1=\sfp_2=0.5$. The saturation rates in both modes are given as $u^1=[1.2,0.7]^T$ and $u^2=[0.2,0.7]^T$.
Thus, both servers have an average saturation rate of 0.7.

\subsubsection{Mode-responsive routing}

For this two-mode system, the policy given by \eqref{eq_phimod} can be parametrized by two constants $\psi_1^1,\psi_1^2\in[0,1]$ (note that admissibility requires $\psi_1^i+\psi_2^i=1$ for $i\in\{1,2\}$).
By Proposition~\ref{prp_bimod}, the routing policy $\phi^{\mathrm{mod}}$ is stabilizing if and only if
\begin{align*}
0.3<(\psi_{1}^1+\psi_{1}^2)/2<0.7,\quad\psi_1^1\in[0,1],\psi_1^2\in[0,1].
\end{align*}
That is, the PDQ system is stable if and only if the average inflows into each server are less than their respective average saturation rate (note that $(\psi_{1}^1+\psi_{1}^2)/2>0.3$ is equivalent to $(\psi_{2}^1+\psi_{2}^2)/2<0.7$).

\subsubsection{Piecewise-affine feedback routing}
Consider the policy given by \eqref{eq_phipwa}.
Admissibility requires $\alpha_{11}=\alpha_{12}$, $\alpha_{21}=\alpha_{22}$, and $\theta_1+\theta_2=1$. Hence, we denote $\alpha_1=\alpha_{11}=\alpha_{12}$ and $\alpha_2=\alpha_{21}=\alpha_{22}$.
For $k=1,2$ and $i\in\{1,2\}$, the expression of the limiting inflows \eqref{eq_varphi} are as follows:
\begin{align*}
&\varphi_{kk}^i=\left\{\begin{array}{ll}
0, & \mbox{if }\alpha_{k}>0,\\
\min\{A,\theta_k\}, & \mbox{if }\alpha_{k}=0,
\end{array}\right.\\
&\varphi_{kh}^i=\left\{\begin{array}{ll}
1, & \mbox{if }\alpha_{h}>0,\\
\min\{A,\theta_k\}, & \mbox{if }\alpha_{h}=0,
\end{array}\right.h\neq k.
\end{align*}
Table \ref{tab_pwa} shows the necessary condition for stability given by Theorem~\ref{Thm_PDQ1} and the sufficient condition for stability given by Proposition~\ref{prp_bpdq}.
\begin{table}[h]
\centering
\caption{Stability conditions (two modes, PWA routing).}
\label{tab_pwa}
\begin{tabular}{|c|c|c|c|}
\hline
$\alpha_1$ & $\alpha_2$     & Necessary condition                     & Sufficient condition                \\ \hline
$=0$ & $=0$ & $0.3\le\theta_1\le0.7$       &  $0.3<\theta_1<0.7$     \\ 
$=0$ & $>0$ & $\theta_1\le0.7$    & $0.3<\theta_1<0.7$     \\ 
$>0$ & $=0$ & $\theta_1\ge0.3$    &  $\theta_1>0.3$ \\ 
$>0$ & $>0$ & $\theta_1\in\real$ &  $\theta_1>0.3$ \\ \hline
\end{tabular}
\end{table}
Note that the restriction on $\theta_k$ is stronger if $\alpha_k=0$.
The intuition is that, if the routing policy is not responsive to the queue length in a server, then an appropriate selection of the nominal inflow $\theta_k$ is crucial to ensure stability.
In addition, the structures of the stability conditions strongly depend on whether $\alpha_k$ is zero, but not on the exact magnitude of $\alpha_k$.
In this example, the gap between the necessary condition and the sufficient condition mainly results from the condition \eqref{eq_phi0<u}, which requires $\theta_1>0.3$.

\subsubsection{Logit routing}

Now, consider the policy \eqref{eq_philog}.
For $k=1,2,\;i\in\{1,2\}$, the limiting inflows are
\begin{subequations}
\begin{align}
&\varphi_{kk}^i=\left\{\begin{array}{ll}
0, & \mbox{if }\beta_{k}>0,\\
\frac{A\exp(\gamma_k)}{\sum_{h=1}^2\exp(\gamma_h)}, & \mbox{if }\beta_{k}=0,
\end{array}\right.\label{eq_varphilog1}\\
&\varphi_{kh}^i=\left\{\begin{array}{ll}
A, & \mbox{if }\beta_{h}>0,\\
\frac{A\exp(\gamma_k)}{\sum_{h=1}^2\exp(\gamma_h)}, & \mbox{if }\beta_{h}=0,
\end{array}\right.h\neq k.
\label{eq_varphilog2}
\end{align}
\label{eq_varphilog}%
\end{subequations}
Again, we can obtain a stability conditions from Theorem~\ref{Thm_PDQ1} and Proposition~\ref{prp_bpdq}.
\begin{table}[h]
\centering
\caption{Stability conditions (two modes, logit routing).}
\label{tab_log}
\begin{tabular}{|c|c|c|c|}
\hline
$\beta_1$ & $\beta_2$     & Necessary condition                     & Sufficient condition                \\ \hline
$=0$ & $=0$ & $|\gamma_1-\gamma_2|\le\log(7/3)$       &
\multirow{4}{*}{$|\gamma_1-\gamma_2|<\log(7/3)$}     \\ 
$=0$ & $>0$ & $\gamma_1-\gamma_2\le\log(7/3)$    & \\ 
$>0$ & $=0$ & $\gamma_1-\gamma_2\ge-\log(7/3)$    &\\ 
$>0$ & $>0$ & $\gamma_1\in\real,\gamma_2\in\real$ & \\ \hline
\end{tabular}
\end{table}
Table~\ref{tab_log} implies that the constants $\gamma_k$ have a stronger impact on stability of the PDQ system than the coefficients $\beta_k$ capturing the sensitivity to queue lengths.
Once again, the gap between the necessary condition and the sufficient condition results from \eqref{eq_phi0<u}, which requires $|\gamma_1-\gamma_2|<\log(7/3)$.

\subsection{A three-mode network}
Suppose that the network has three modes $\{1,2,3\}$ with symmetric transition rates $\lambda_{ij}=1$ for all $i,j\in\I$. Thus, the steady-state probabilities are $\sfp_1=\sfp_2=\sfp_3=1/3$. The saturation rates in the three modes are $u^1=[1.2,0.7]^T$, $u^2=[0.7,0.7]^T$, and $u^3=[0.2,0.7]^T$; i.e. the average saturation rates are equal to those in the two-mode case.
The main difference between the analysis in this subsection and that in the previous subsection is that the sufficient conditions for stability below are obtained numerically (in terms of solving the BMI \eqref{eq_BMI}) instead of analytically.

\subsubsection{Mode-responsive routing}

For ease of presentation, we assume that $\psi_k^2=\psi_k^3$ for $k\in\{1,2\}$.
The limiting inflows $\varphi_{kh}^i$ are given by
\begin{align*}
\varphi_{kh}^i=\psi_k^i,\;h\in\{1,2\},k\in\{1,2\},i\in\I.
\end{align*}
Theorem~\ref{Thm_PDQ1} gives a necessary condition for stability:
\begin{align}
0.3\le1/3\psi_1^1+2/3\psi_1^2\le0.7,
\label{eq_mode}
\end{align}
whose complement is the ``Unstable'' region in Figure~\ref{fig_stable}.
Figure~\ref{fig_stable} also shows a ``Stable'' region obtained from Theorem~\ref{Thm_PDQ2}; the BMI \eqref{eq_BMI} is solved using YALMIP \cite{lofberg04}.
In contrast to the two-mode case, there is an ``Unknown'' region between the ``Stable'' and ``Unstable'' regions, due to the gap between the necessary condition (Theorem~\ref{Thm_PDQ1}) and the sufficient condition (Theorem~\ref{Thm_PDQ2}).

\begin{figure}[hbt]
\centering
\includegraphics[width=0.3\textwidth]{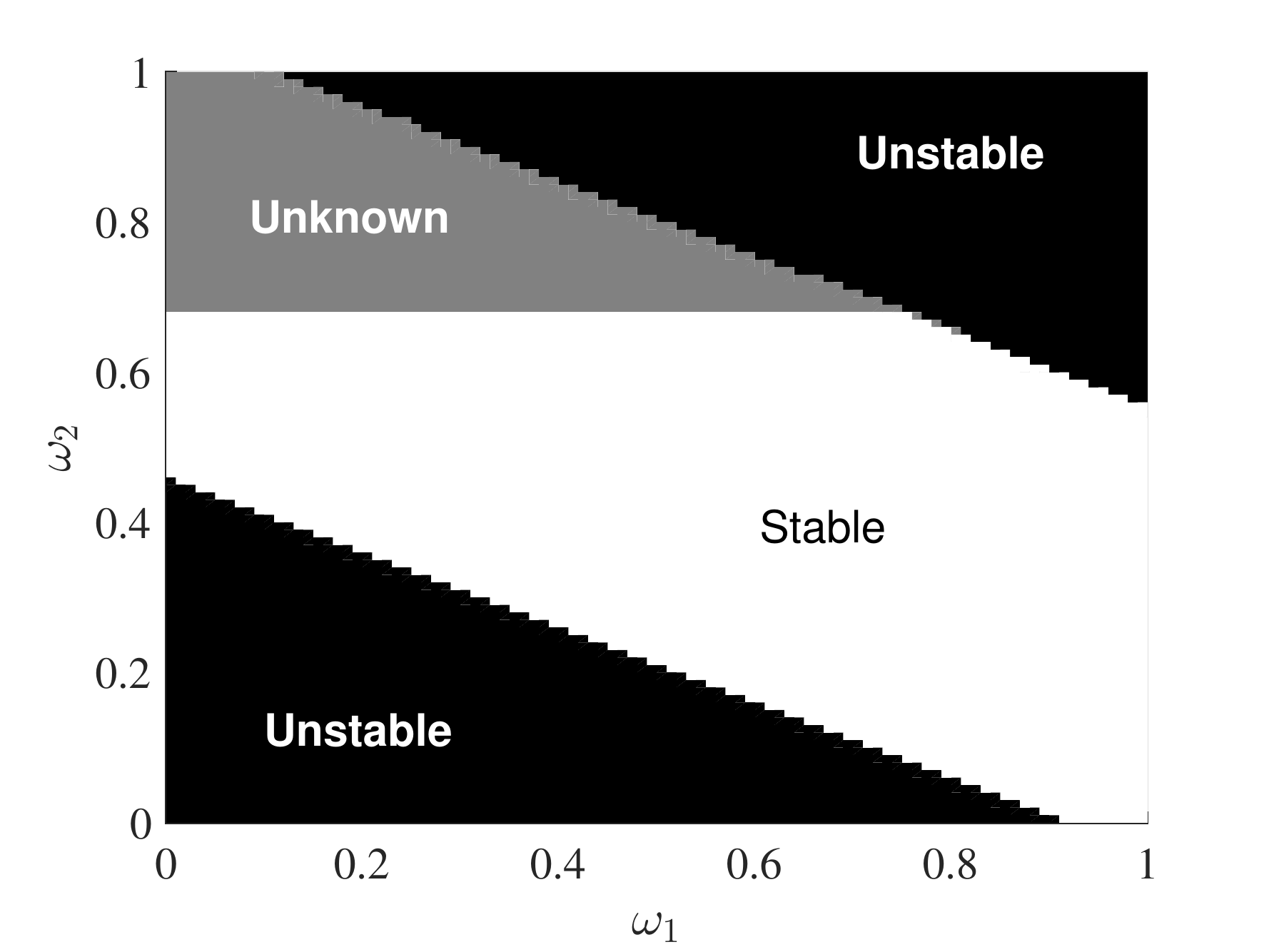}
\caption{Stability of various $(\psi_1^1,\psi_1^2)$ pairs.}
\label{fig_stable}
\end{figure}

\subsubsection{Queue-responsive routing policies}
For the piecewise-affine routing policy \eqref{eq_phipwa} and the logit routing policy \eqref{eq_philog}, 
\begin{table}[h]
\centering
\caption{Stability conditions (three modes, PWA routing).}
\label{tab_pwa2}
\begin{tabular}{|c|c|c|c|}
\hline
$\alpha_1$ & $\alpha_2$     & Necessary condition                     & Sufficient condition                \\ \hline
$=0$ & $=0$ & $0.3\le\theta_1\le0.7$       &  $0.41\le\theta_1\le0.59$     \\ 
$=0$ & $>0$ & $\theta_1\le0.7$    & $0.41\le\theta_1\le0.59$     \\ 
$>0$ & $=0$ & $\theta_1\ge0.3$    &  $\theta_1\ge0.36$ \\ 
$>0$ & $>0$ & $\theta_1\in\real$ &  $\theta_1>0.3$ \\ \hline
\end{tabular}
\end{table}
\begin{table}[h]
\centering
\caption{Stability conditions (three modes, logit routing).}
\label{tab_log2}
\begin{tabular}{|c|c|c|c|}
\hline
$\beta_1$ & $\beta_2$     & Necessary condition                     & Sufficient condition                \\ \hline
$=0$ & $=0$ & $|\gamma_1-\gamma_2|\le\log(7/3)$       &
\multirow{4}{*}{$|\gamma_1-\gamma_2|\le\log1.7$}     \\ 
$=0$ & $>0$ & $\gamma_1-\gamma_2\le\log(7/3)$    & \\ 
$>0$ & $=0$ & $\gamma_1-\gamma_2\ge-\log(7/3)$    &\\ 
$>0$ & $>0$ & $\gamma_1\in\real,\gamma_2\in\real$ & \\ \hline
\end{tabular}
\end{table}
Tables \ref{tab_pwa2} and \ref{tab_log2} show the stability conditions.
In comparison to the two-mode case, the necessary conditions are unchanged, but the sufficient conditions in the three-mode case are more restrictive.
This indicates that the sufficient condition becomes more restrictive as the number of modes (and thus the number of bilinear inequality constraints) increases.
\section*{Acknowledgments}
This work was supported by NSF CNS-1239054 CPS Frontiers, NSF CAREER Award CNS-1453126, and AFRL Lablet-Secure and Resilient Cyber-Physical Systems. 
We are deeply grateful for the insightful comments from the three anonymous reviewers and the associate editor.

\bibliographystyle{ieeetr}
\bibliography{Bibliography}
\end{document}